\documentclass[a4paper,10pt]{article}

\def\ifundefined#1{\expandafter\ifx\csname#1\endcsname\relax}

\usepackage{graphicx}
\usepackage{graphics}
\usepackage{mathrsfs}
\usepackage{float}
\usepackage{rotating}

\usepackage[T1]{fontenc}
\usepackage[english]{babel}
\usepackage[utf8]{inputenc}
\usepackage{amsthm,amsmath,amssymb,upref,amscd,amsfonts}
\usepackage{euscript}
\usepackage{epsfig}
\usepackage{ae}
\usepackage{aecompl}
\usepackage{esint}
\usepackage[metapost]{mfpic}
\usepackage{calc}
\usepackage{enumerate}
\usepackage{enumitem}
\usepackage{url}
\usepackage{bm}
\usepackage[affil-it]{authblk}

\theoremstyle{plain}
\newtheorem{theorem}{Theorem}[section]

\newtheorem{lemma}[theorem]{Lemma}
\newtheorem{corollary}[theorem]{Corollary}
\newtheorem{theorem*}{Theorem}

\theoremstyle{definition}

\numberwithin{equation}{section}

\mathchardef\sa="303A
\renewcommand{\epsilon}{\varepsilon}

\newcommand{\ip}[2]{\ensuremath{( {#1}, \, {#2} )}}
\newcommand{\Lip}[2]{\ensuremath{\langle {#1}, \, {#2} \rangle }}

\newcommand{\lm}[1][m]{\ensuremath{\lambda_{#1}}}
\newcommand{\um}[1][m]{\ensuremath{\mu_{#1}}}
\newcommand{\tk}[1][k]{\ensuremath{\tau_{#1}}}
\newcommand{\eps}{\ensuremath{\epsilon}}
\newcommand{\Jm}[1][m]{\ensuremath{J_{#1}}}

\newcommand{\kk}[1][k]{\ensuremath{\kappa_{#1}}}

\newcommand{\Llm}[1][m]{\ensuremath{\Lambda_{#1}(\Oe)}}
\newcommand{\Lum}[1][m]{\ensuremath{\Lambda_{#1}(\Ot)}}

\newcommand{\vp}{\ensuremath{\varphi}}
\newcommand{\vs}{\ensuremath{\psi}}

\newcommand{\Pvp}{\ensuremath{\Phi}}

\newcommand{\Vp}{\ensuremath{\Psi_{\vp}}}
\newcommand{\Vs}{\ensuremath{\Psi_{\vs}}}

\newcommand{\Rvp}{\ensuremath{R_{\varphi}}}
\newcommand{\Rvpo}{\ensuremath{R_{\varphi}^{(1)}}}
\newcommand{\Rvpt}{\ensuremath{R_{\varphi}^{(2)}}}

\newcommand{\Rvs}{\ensuremath{R_{\psi}}}

\renewcommand{\S}{\ensuremath{S}}
\newcommand{\T}{\ensuremath{T}}
\newcommand{\Se}{\ensuremath{S_1}}
\newcommand{\St}{\ensuremath{S_2}}
\newcommand{\Sj}{\ensuremath{S_j}}

\newcommand{\dist}{\ensuremath{\mbox{dist}}}

\newcommand{\rr}{\ensuremath{\rho}}

\newcommand{\N}{\ensuremath{N}}

\renewcommand{\H}{\ensuremath{H}}
\newcommand{\He}{\ensuremath{H_1}}
\newcommand{\Ht}{\ensuremath{H_2}}
\newcommand{\Hj}[1][j]{\ensuremath{H_{j}}}

\newcommand{\Xj}[1][j]{\ensuremath{X_{#1}}}

\newcommand{\Om}{\ensuremath{\Omega}}

\newcommand{\Oe}{\ensuremath{\Omega_1}}
\newcommand{\Ot}{\ensuremath{\Omega_2}}
\newcommand{\Oet}{\ensuremath{\Omega_{12}}}

\newcommand{\Ge}{\ensuremath{\Gamma_1}}
\newcommand{\Gt}{\ensuremath{\Gamma_2}}
\newcommand{\Get}{\ensuremath{\Gamma_{12}}}

\newcommand{\Dd}{\ensuremath{D}}

\newcommand{\C}{\ensuremath{\mathcal{C}}}
\newcommand{\Ck}{\ensuremath{\mathcal{C}_k}}

\newcommand{\ndiv}{\ensuremath{\partial_{\nu}}}

\newcommand{\nablat}{\ensuremath{\nabla_{\tau}}}

\title{Hadamard Asymptotics for Eigenvalues of the Dirichlet Laplacian}
\author{Vladimir Kozlov}
\author{Johan Thim\footnote{Corresponding author: {\tt johan.thim@liu.se}}}
\affil{\small Department of Mathematics, University of Link\"{o}ping, Link\"{o}ping, Sweden}

\date{\today}

\begin{document}

\maketitle

\begin{abstract}
\noindent
This paper is dedicated to the classical Hadamard formula for asymptotics of eigenvalues
of the Dirichlet-Laplacian under perturbations of the boundary. We prove that the Hadamard formula
still holds for~$C^1$-domains with~$C^1$-perturbations. We also derive an optimal estimate for the
remainder term in the~$C^{1,\alpha}$-case. Furthermore, if the boundary is merely Lipschitz, we show
that the Hadamard formula is not valid.
 
\bigskip


\medskip


\end{abstract}


\section{Introduction}\label{SK1}
The question of how eigenvalues change when the domain is slightly perturbed is a classical problem
probably going at least as far back as Rayleigh~\cite{Rayleigh1877}, who studied eigenvalues and domain perturbation
in connection with acoustics as early as in the nineteenth century.
The approach given in this article owes to results by Hadamard~\cite{Hadamard1968}, who in the early 20th century
studied perturbations of domains with smooth boundary,
where the perturbed domain~$\Omega_{\epsilon}$ is represented by~$x_{\nu} = \epsilon h(x')$
where~$x' \in \partial \Omega_0$,~$x_{\nu}$ is the signed distance to the boundary ($x_{\nu} < 0$ for~$x \in \Omega_0$),~$h$ is a smooth
function,~$\epsilon$ is a small parameter, and~$\Omega_0$ is the reference domain. Hadamard's result
for the first eigenvalue of the Dirichlet-Laplacian is given by
\begin{equation}\label{A4a}
\Lambda(\Omega_{\epsilon}) = \Lambda(\Omega_0) -
	\epsilon \int_{\partial \Omega_0} h  |\nabla \vp|^2  dS
	+ o(\epsilon),
\end{equation}
where $dS$ is the surface measure on~$\partial \Omega_0$ and~$\vp$ is an eigenfunction corresponding
to~$\Lambda(\Omega_0)$ such that~$\| \vp \|_{L^2(\Omega_0)} = 1$.
It is worth noting, that the problem of how eigenvalues change when the domain is perturbed, is closely
related to shape optimization. We refer here
to Henrot~\cite{Henrot2006}, and Soko\l{}owski and Zol\'esio~\cite{Sokolowski1992}, and references found therein.

The aim of this article is to  find minimal assumptions on the smoothness of the boundary when the 
Hadamard formula is still valid.
A large quantity of studies of the Dirichlet problem already exists in the literature; see, for instance,
Grinfeld~\cite{Grinfeld2010}, Henrot~\cite{Henrot2006}, Kozlov~\cite{Kozlov2013,Kozlov2006},
Kozlov and Nazarov~\cite{Kozlov2012}, and references found therein. In this article,
we present an asymptotic formula of Hadamard type for perturbations in the case when the
domains are of class~$C^1$ or~$C^{1,\alpha}$, respectively. 
The first class of domains is optimal for validity of the Hadamard formula. For the second class of domains,
we give an optimal estimate of the remainder term.
Let~$\Oe$ and~$\Ot$ be bounded domains in~$\mathbf{R}^n$ with boundaries~$\Ge$ and~$\Gt$, respectively.
We consider the spectral problems
\begin{equation}
\label{eq:maineq}
\left\{
\begin{aligned}
 -\Delta  u &= \Llm[{}] u   & & \mbox{in } \Oe,\\
u           &=  0 & & \mbox{on } \Ge
\end{aligned} \right.
\end{equation}
and
\begin{equation}
\label{eq:maineq2}
\left\{
\begin{aligned}
 -\Delta v &= \Lum[{}] v  & & \mbox{in } \Ot,\\
v          &=  0 & & \mbox{on } \Gt.
\end{aligned} \right.
\end{equation}
In the case when the boundary is nonsmooth, we consider the corresponding weak formulation of the problem
on the Hilbert spaces~$H^1_0(\Oe)$ and~$H^1_0(\Ot)$ with the usual inner product.
Note though, that the techniques used are applicable to a wider class of partial differential operators.
In particular to uniformly elliptic operators of second order.

It is known that if the two domains are close enough, both problems have the same number of
eigenvalues in a small enough neighborhood of~$\Llm[{}]$ as the multiplicity of the eigenvalue.
This means that for a fixed eigenvalue~$\Llm[{}]$ of~(\ref{eq:maineq}) of multiplicity~$m$,
there are precisely~$m$ eigenvalues~$\Lum[{}]$ of~(\ref{eq:maineq2}) (counting multiplicity) near~$\Llm[{}]$.
This is a consequence of
the continuous dependence of eigenvalues on the domain;
see, e.g., Kato~\cite{Kato1966} (Sections~IV.3 and~V.3) or Henrot~\cite{Henrot2006} and references therein.
The explicit result in terms of quantities used in this article can be found in Kozlov~\cite{Kozlov2013} (Proposition~3).
We will denote by~$\Xj[k] \subset H^1_0(\Oe)$ the eigenspace corresponding to the eigenvalue~$\Llm[k]$ and denote
the dimension of~$\Xj[k]$ by~$\Jm[k]$.
For our results, we will characterize how close the two domains are in the sense
of the Hausdorff distance between the sets~$\Oe$ and~$\Ot$, i.e.,
\begin{equation}
\label{eq:d_def}
d = \max \{ \sup_{x \in \Oe} \inf_{y \in \Ot} |x-y|, \; \sup_{y \in \Ot} \inf_{x \in \Oe} |x-y| \}.
\end{equation}
We do not assume that one domain is a subdomain of the other. It should be noted however,
that the abstract result presented below in Section~\ref{s:prelim} permits a more general type of proximity quantity
for the two domains.

We consider three cases of regularity of the boundary~$\Ge$, namely $C^{1,\alpha}$, $C^1$,
and Lipschitz boundaries. Let us first consider the Lipschitz case.
Then there exists a positive constant~$M$ such that the boundary $\Ge$ can be covered by
a finite number of cylinders $\Ck$,~$k=1,2,\ldots,N$, where there exists orthogonal coordinate systems in which
\[
\Ck \cap \Oe= \Ck \cap \{ y = (y', y_n) \,:\, y_n > h_k^{(1)}(y')\},
\]
where the center of $\Ck$ is at the origin 
and
\[
\Ck = \{ (y',y_n)\,:\, y' \in B'_{r_k}(0) ,\, |y_n| < \delta_k \}. 
\]
Here,~$B'_{k}=B'_{r_k}(0)$ is the $(n-1)$-dimensional ball of radius $r_k$ and with the center $0$. 
We assume that~$h_k^{(1)}(0) = 0$ and that~$h_k^{(1)}$ are Lipschitz functions, i.e.,
\[
|h_k^{(1)}(y') - h_k^{(1)}(z')| \leq M |y' - z'|.
\]
This class of domains defined by a constant~$M$ and cylinders~$\Ck$,~$k=1,2,\ldots,N$, will be denoted in what follows by
${\mathcal L}={\mathcal L}(M,\C_1,\ldots,\C_N)$.

We assume that~$\Ot$ is close to~$\Oe$ in the sense that~$\Ot \in {\mathcal L}(M,\C_1,\ldots,\C_N)$ can be described by
\[
\Ck \cap \Ot = \Ck \cap \{ y = (y', y_n) \,:\, y_n > h_k^{(2)}(y')\},
\]
where~$h_k^{(2)}$ are also Lipschitz continuous with Lipschitz constant~$M$ and
\[
\widehat{d} = \max_{k=1,2,\ldots,N} \sup \{ | h_k^{(1)}(y') - h_k^{(2)}(y')| \,:\, y'\in B'_k \}
\]
is assumed to be small. One can show that there exists positive constants~$c_1$ and~$c_2$, 
such that~$c_1 \widehat{d} \leq d \leq c_2 \widehat{d}$.

The case when $\Oe$ is a $C^1$- or $C^{1,\alpha}$-domain is defined analogously, with the following additional
assumptions.

\noindent {\bf $\bm{C^1}$-assumption.} We assume that $h_k^{(1)} \in C^{1}(\overline{B'_k})$ such that
\begin{equation}\label{K2a}
h_k^{(1)}(0) = \partial_{x_i} h_k^{(1)}(0) = 0, \quad i=1,2,\ldots,n-1,
\end{equation}
and
\begin{equation}
\label{eq:grad_hk_c1}
| \nabla ( h_k^{(1)} - h_k^{(2)}) | = o(1), \quad \mbox{as } d \rightarrow 0.
\end{equation}

\noindent {\bf $\bm{C^{1,\alpha}}$-assumption.} We assume that~$h_k^{(1)} \in C^{1,\alpha}(\overline{B'_k})$ 
and that~(\ref{K2a}) holds. Furthermore,
\begin{equation}
\label{eq:grad_hk_holder}
| \nabla ( h_k^{(1)} - h_k^{(2)}) | \leq C d^{\alpha}.
\end{equation}

Note that $h_k^{(2)}$ are only assumed to be Lipschitz continuous in both cases and satisfy~(\ref{eq:grad_hk_c1}) 
or~(\ref{eq:grad_hk_holder}), respectively.

Let us define the function~$\sigma$ on the surface $\Ge$ in the case of the $C^1$-assumption on the boundary. 
Let $\nu = \nu(P)$ be the unit outward normal to the boundary~$\Ge$ at the point $P$. 
For $P \in \Ge$ we introduce the number $\sigma = \sigma(P)$. It is the smallest (in the absolute value) root 
of the equation
\begin{equation}
\label{KK1a}
P + \sigma \nu = Q \in \Gt,
\end{equation}
where~$Q$ is the nearest point on~$\Gt$ to~$P$ lying on the line passing through~$P$ with the direction~$\nu$. 
Clearly~$\sigma$ is positive if~$Q$ is outside~$\Oe$ and negative if~$Q$ is inside. 
One can verify that this function is Lipschitz continuous on~$\Ge$ with Lipschitz constant 
depending on~$M$ and~$\C_1,\ldots,\C_N$.

\begin{theorem}
\label{i:t:holder_boundary}
Suppose that~$\Oe$ is a~$C^{1,\alpha}$-domain with~$0 < \alpha < 1$ and that~$\Ot$ is as described above.
Then
\begin{equation}
\label{i:eq:holder_main_boundary}
\begin{aligned}
\Lum[k] - \Llm[m] = {} &
\kk + O(d^{1+\alpha})
\end{aligned}
\end{equation}
for every~$k = 1,2,\ldots,\Jm[m]$.
Here~$\kk[{}] = \kk$ is an eigenvalue of the problem
\begin{equation}
\label{eq:T2_tk_intro}
\kk[{}] \int_{\Oe} \nabla \vp \cdot \nabla \vs \, dx =
\int_{\Ge} \sigma \, 
	\nabla \vp \cdot \nabla \vs
	\,  dS
\quad \mbox{for all } \vs \in \Xj[m],
\end{equation}
where~$\vp \in \Xj[m]$. Moreover, $\kk[1],\kk[2],\ldots,\kk[\Jm]$ in~{\rm(}\ref{i:eq:holder_main_boundary}{\rm)}
run through all eigenvalues of~{\rm(}\ref{eq:T2_tk_intro}{\rm)} counting their multiplicities.
\end{theorem}

\noindent
Observe that~(\ref{eq:T2_tk_intro}) can be phrased as a spectral problem on the Hilbert space~$\Xj[m]$ by using the
Riesz representation theorem of the operator on the right-hand side. Here $dS$ is the surface measure.

\begin{theorem}
\label{i:t:C1_boundary}
Suppose that~$\Oe$ is a~$C^{1}$-domain, and~$\Ot$ is as described above.
Then
\begin{equation}
\label{i:eq:C1_main_boundary}
\begin{aligned}
\Lambda_k(\Ot) - \Lambda_m(\Oe) = \kappa_k +  o(d)
\end{aligned}
\end{equation}
for every~$k = 1,2,\ldots,J_m$.
\end{theorem}

\noindent
If~$\Oe$ and~$\Ot$ are two Lipschitz domains, then we can prove the estimate
\begin{equation}
\label{JiJ}
|\Lambda_k(\Ot) - \Lambda_m(\Oe)| \leq C d
\end{equation}
for every~$k = 1,2,\ldots,J_m$.

We note that if~$u$ satisfies (\ref{eq:maineq}) for a Lipschitz domain~$\Oe$,
then the normal derivative~$\partial_{\nu} u|_{\Ge} \in L^2(\Ge)$ and so the integral in the
right-hand side of~(\ref{eq:T2_tk_intro}) is well defined. Despite this fact,
we have demonstrated in Section~\ref{SA4a} that the formula presented in~(\ref{i:eq:holder_main_boundary}) 
does not hold for Lipschitz domains and their small perturbations. Indeed, we show in Section~\ref{SA4a} 
that the remainder~$O(d^{1+\alpha})$ is optimal for~$C^{1,\alpha}$-perturbations.
The result similar to Theorem \ref{i:t:holder_boundary}
for the Neumann problem is proved in~\cite{Kozlov2016} and an analogue of Theorem \ref{i:t:C1_boundary} 
for the Neumann problem, albeit with a more complicated expression for the leading term~$\kappa_k$, is
derived in~\cite{Thim2015}.

\bigskip

The paper is organized as follows. In Section~\ref{s:prelim} we present an abstract approach 
developed in \cite{Kozlov2013}, which concerns the asymptotics of eigenvalues of unbounded 
operators when their domains are changed. As a result of this approach, a theorem on asymptotics 
of eigenvalues is presented. It involves different terms and these terms are estimated in the remaining 
parts of the paper. The analysis is based on the theory of elliptic boundary value problems in
Lipschitz, $C^1$- and~$C^{1,\alpha}$-domains that was developed in particular in the works of Dahlberg, 
Fabes, Jodeit, Kenig and Rivi\'{e}re. In the last section, we present an example which demonstrates the
sharpness of Theorems~\ref{i:t:holder_boundary} and~\ref{i:t:C1_boundary}.


\section{Preliminary Results and Definitions}
\label{s:prelim}

Here we present an abstract result from \cite{Kozlov2013}, which will play an important 
role in the proofs of Theorems~\ref{i:t:holder_boundary} and~\ref{i:t:C1_boundary}.

We suppose that~$\Oe$ and~$\Ot$ are open Lipschitz domains and that~$\Dd$ is an open ball such
that~$\Oe \subset \Dd$ and~$\Ot \subset \Dd$. Put~$\He = H^1_0(\Oe)$,~$\Ht = H^1_0(\Ot)$ and finally~$\H = H^1_0(\Dd)$. 
We extend functions by zero outside their respective domains.
We let~$\ip{\cdot}{\cdot}$ and~$\Lip{\cdot}{\cdot}$ denote the inner products on~$\H$ given by
\[
\ip{u}{v} = \int_{\Dd} \nabla u \cdot \nabla v \, dx
\qquad \mbox{and} \qquad
\Lip{u}{v} = \int_{\Dd} u \, v \, dx,
\]
respectively. Moreover, let~$\| \, \cdot \, \|$ and~$| \, \cdot \, |$ be the norms induced by~$\ip{\cdot}{\cdot}$
and~$\Lip{\cdot}{\cdot}$, respectively. We will consider the spectral problems
\begin{equation}
\label{eq:maineig}
\ip{\vp}{v} = \lambda \Lip{\vp}{v} \quad \mbox{for every~$v \in \He$}
\end{equation}
and
\begin{equation}
\label{eq:perturbeig}
\ip{U}{V} = \mu \Lip{U}{V} \quad \mbox{for every~$V \in \Ht$.}
\end{equation}
We enumerate the eigenvalues~$\Llm[k] = \lm[k]$,~$k=1,2,\ldots$, of~(\ref{eq:maineq}) increasingly
according to~$0 < \lm[1] < \lm[2] < \cdots$.
Let~$\Xj[k] \subset \He$ be the eigenspace corresponding to the eigenvalue~$\lm[k]$ and denote
the dimension of~$\Xj[k]$ by~$\Jm[k]$. 
In this article we study eigenvalues of~(\ref{eq:perturbeig}) located in a neighborhood of~$\lm$,
where~$m$ is fixed.
Note that it is known that there are precisely~$\Jm[m]$ eigenvalues of~(\ref{eq:maineq2}) near~$\lm$;
see, e.g., Lemma~3.1 in~\cite{Kozlov2013}.
We denote these eigenvalues by~$\um[1],\um[2],\ldots,\um[{\Jm[m]}]$.

Let~$\Sj \colon \H \rightarrow \Hj$ be the orthogonal projectors with image~$\Hj$ for~$j=1,2$.
For simplicity, we denote~$\St$ by~$\S$. Furthermore, put~$\T v = v - \S v$ for~$v \in \He$.
We will employ results from \cite{Kozlov2013}, where an abstract framework for domain dependence
of Dirichlet eigenvalues was presented. To measure the closeness of two domains, the norm difference between
the projectors~$\Se$ and~$\St$ is used:
\begin{equation}
\label{eq:sigma}
| (\Se - \St) u |^2 \leq \eps \| u \|^2, \quad u \in \H.
\end{equation}
The best constant~$\eps$ measures the proximity of the
spaces~$\He$ and~$\Ht$, and therefore the closeness of~$\Oe$ and~$\Ot$.
One can show that~$\eps \approx d$ in our case.

To formulate an abstract asymptotic result, we will use the solution~$\Vp \in \Ht$ to the problem
\begin{equation}
\label{eq:def_Vp}
\ip{\Vp}{w} = \ip{\vp}{w} - \lm \Lip{\vp}{w} \quad \mbox{ for every~$w \in \Ht$}.
\end{equation}
Theorem~1 in \cite{Kozlov2013} can be reformulated as follows.

\begin{theorem}
\label{t:kozlov2013}
The asymptotic formula
\begin{equation}\label{M28s}
\um[k]^{-1} = \lm[m]^{-1} + \tk + O(\rr + |\tk|\eps), \quad k=1,2,\ldots,J_m,
\end{equation}
holds, where
\begin{equation}
\label{eq:rho}
\rr = \max_{\vp \in \Xj[m], \; \| \vp \| = 1} \bigl( |\T \vp |^2 + |\Vp|^2 + \eps \| \Vp \|^2 \bigr).
\end{equation}
Here~$\tau = \tk$ is an eigenvalue of the problem
\begin{equation}
\label{eq:main_asymp_D}
\lm^{-1} \bigl( \ip{\Vp}{\Vs} - \ip{\T\vp}{\T\vs} - \ip{\Vp}{\vs} - \ip{\vp}{\Vs} \bigr) = \tau \ip{\S \vp}{\S \vs},
\end{equation}
where~$\vp \in \Xj[m]$ and~$\tau_1,\tau_2,\ldots,\tau_{j_m}$ run through all eigenvalues 
of~{\rm(}\ref{eq:main_asymp_D}{\rm)} counted with multiplicity.
\end{theorem}

\section{Estimates for the functions $\varphi$, $\Psi_\varphi$ and $T\varphi$}
\label{s:def_lipdom}

In this section we give estimates and some representations for the terms appearing in (\ref{eq:rho})
and (\ref{eq:main_asymp_D}). All of them are valid for Lipschitz domains.

We will use the following notation.
Let~$\Om$ be a Lipschitz domain.
The truncated cones~$\Gamma(x')$ at~$x' \in \partial \Om$ are given by, e.g.,
\[
\Gamma(x') = \{ x \in \Om \sa  |x - x'| < 2 \mbox{dist}(x, \partial \Om) \}
\]
and the non-tangential maximal function is defined on the boundary~$\partial \Om$ by
\[
\N(u)(x') = \max_{k=1,2,\ldots,N} \sup \{ |u(x)| \sa x \in \Gamma(x') \cap {\mathcal C}_k \}
\]
and
$$
\N(\nabla u)(x') = \max_{k=1,2,\ldots,N} \sup \{ |\nabla u(x)| \sa x \in \Gamma(x') \cap {\mathcal C}_k \}
$$
We refer to Kenig~\cite{Kenig1994} for further details.

\noindent We will use the short-hand notation~$\nablat u$ for the tangential gradient.

From the definition of $\Vp$ in~(\ref{eq:def_Vp}), it follows that $\Vp$ is harmonic in 
the domains~$\Oe \cap \Ot$ and~$\Ot \setminus \overline{\Oe}$. We will use the representation
\begin{equation}
\label{eq:rep_Vp}
\Vp = \vp - \Rvp,
\end{equation}
where~$\Rvp \in \Ht$ solves the equation~$\ip{\Rvp}{w} = \lm \Lip{\vp}{w}$ for all~$w \in \Ht$, or, equivalently,
\[
-\Delta \Rvp =\lm \vp \;\;\mbox{in $\Ot$}\;\;\mbox{and}\;\; \Rvp = 0\;\;\mbox{on $\Gt$.}
\]
We represent $\Rvp$ as $\Rvp = \Rvpo + \Rvpt$, where
\[
-\Delta \Rvpo = \lm \vp \;\;\mbox{in $\Dd$}\;\;\mbox{and}\;\; \Rvpo = 0 \;\;\mbox{on $\partial \Dd$}
\]
and
\[
\Delta \Rvpt = 0 \;\;\mbox{in $\Ot$}\;\;\mbox{and}\;\; \Rvpt = -\Rvpo \;\;\mbox{on $\Gt$.}
\]
Since $\vp \in H^1(\Dd)$ it follows that $\Rvpo \in H^3(\Dd)$ and
\begin{equation}
\label{M7a}
\| \Rvpo \|_{H^3(\Dd)} \leq C \| \vp \|_{H^1(\Dd)}.
\end{equation}
Moreover, from, e.g., Theorems 1 and~3 in Jerison and Kenig~\cite{jerisonkenig81N}, we obtain that
\begin{equation}
\label{eq:N_Rvpo}
\int_{\Gt} \left( | N(\Rvpt) |^2 + | N(\nabla \Rvpt) |^2 \right) dS \leq
C \| \Rvpo \|_{H^{1}(\Gt)}^2 \leq C\|\vp\|_{H^1(\Dd)}^2.
\end{equation}
From the representation (\ref{eq:rep_Vp}) one can see that the normal derivative of the 
function~$\Vp$ has a jump on $\Ge$ since the function $\Rvp$ is of class $H^3$ in a
neighborhood of~$\Ge$.

\newcommand{\Lb}{\ensuremath{L^2}}

To make the notation more compact, we let
\[
\Oet = \Oe \cap \Ot, \;\; \Get = \partial \Oet \;\; \mbox{and} \;\; \Lb = L^2(\Get).
\]
Then the boundary~$\Get$ is given by the relation~$y_n=h^{(12)}_k(y')$ in each cylinder~${\mathcal C}_k$, 
where~$h^{(12)}_k(y')=\max (h^{(1)}_k(y'),h^{(2)}_k(y'))$. One can verify that this function
also satisfies~(\ref{eq:grad_hk_c1}) in the~$C^1$-case and~(\ref{eq:grad_hk_holder}) in
the~$C^{1,\alpha}$-case. In what follows we shall use the short-hand notation~$h_1$,~$h_2$ and~$h_{12}$
for the functions~$h^{(1)}_k$,~$h^{(2)}_k$ and~$h^{(12)}_k$, respectively.

Let~$\nu = \nu(P)$ denote the outwards unit normal vector at the point~$P$ to the boundary $\Ge$ ($\Gt$).
The normal derivative~$\ndiv u$ at~$P \in \Ge$ ($P \in \Gt$) is defined by~$\ndiv u = \nabla u \cdot \nu$.
We can now formulate the following lemmas concerning the inner products in~(\ref{eq:main_asymp_D}).

\begin{lemma}
\label{l:VpVs}
Let~$\Oe$ and~$\Ot$ be Lipschitz domains from~${\mathcal L}(M,\C_1,\ldots,\C_N)$.
Then
\begin{equation}
\label{eq:est_VpVs}
\left| \ip{\Vp}{\Vs} - \int_{\Ot \setminus \Oe} \nabla \Rvp \cdot \nabla \Rvs \, dx \right|
\leq
	  C \| \Psi_\varphi \|_{\Lb}
	  \| \nabla_\tau \Psi_\psi \|_{\Lb}
\end{equation}
for all~$\vp,\vs \in \Xj[m]$.
\end{lemma}

\begin{proof}
Since $\varphi$ and $\psi$ vanish outside $\Oe$, we have
\[
\ip{\Vp}{\Vs} 
=  \int_{\Ot \setminus \Oe} \nabla \Rvp \cdot \nabla \Rvs \, dx
+ \int_{\Oet} \nabla \Vp \cdot \nabla \Vs \, dx.
\]
Using the relation
\[
\int_{\Oet} \nabla \Vp \cdot \nabla \Vs \, dx
=
\int_{\Get} \Vp \ndiv \Vs \, dS
\]
and the inequality
\begin{equation}
\label{M19a}
\| \ndiv \Vs  \|_{\Lb}
\leq C
\| \nablat \Vs \|_{\Lb},
\end{equation}
we obtain that
\[
\left| \int_{\Oet} \nabla \Vp \cdot \nabla \Vs \, dx \right|
\leq C \| \Vp \|_{\Lb} \| \nablat \Vs \|_{\Lb},
\]
which in turn shows that~(\ref{eq:est_VpVs}) holds.
\end{proof}

\begin{lemma}
\label{l:vpVs}
Suppose that~$\Oe$ and~$\Ot$ are the same domains as in Lemma \ref{l:VpVs}.
Then
\begin{equation}
\label{eq:est_vpVs}
\bigl| \ip{\vp}{\Vs} \bigr| \leq C
\| \vp \|_{\Lb}
	\| \nablat \Vs \|_{\Lb}
\end{equation}
for all~$\vp,\vs \in \Xj[m]$.
\end{lemma}

\begin{proof}
Since $\Vs$ is harmonic in $\Oet$, we have
\[
\ip{\vp}{\Vs} = \int_{\Get} \vp \ndiv \Vs \, dS.
\]
Using (\ref{M19a}) we arrive at (\ref{eq:est_vpVs}).
\end{proof}

\noindent Note that Lemma~\ref{l:vpVs} immediately implies that
\begin{equation}
\label{eq:est_Vpvs}
\bigl| \ip{\Vp}{\vs} \bigr| \leq C
\| \vs \|_{\Lb}
	\| \nablat \Vp  \|_{\Lb}
\end{equation}
for all~$\vp,\vs \in \Xj[m]$.

\begin{lemma}
\label{l:est_T}
Suppose that~$\Oe$ and~$\Ot$ are the same domains as in Lemma \ref{l:VpVs}.
Then
\begin{equation}
\label{eq:est_T}
\left| \ip{\T\vp}{\T\vs} - \int_{\Oe \setminus \Ot} \nabla \vp \cdot \nabla \vs \, dx \right|
\leq C \| \vs \|_{\Lb} \| \nablat \vp \|_{\Lb}
\end{equation}
for all~$\vp,\vs \in \Xj[m]$.
\end{lemma}

\begin{proof}
Let~$\Pvp = \T \vp = \vp - \S \vp$ and $\Psi = T\vs$. Since $\S \vp \in H_2$ and $\Phi$ is orthogonal to~$\Ht$, 
we have $\Phi = \vp$ in $\Oe \setminus \Ot$ and that~$\Phi$ is harmonic in $\Oet$. Therefore
\[
\int_{\Oe} \nabla \Phi \cdot \nabla \Psi \, dx = 
	\int_{\Oe \setminus \Ot} \nabla\Phi \cdot \nabla \Psi \, dx 
		+ \int_{\Oet} \nabla\Phi \cdot \nabla \Psi \,  dx.
\]
Since
\begin{equation*}
\int_{\Oet} \nabla \Pvp \cdot \nabla \Psi \, dx  =
	 \int_{\Get} \Psi \ndiv \Pvp \, dS,
\end{equation*}
we can use~(\ref{M19a}), which implies that~(\ref{eq:est_T}) holds.
\end{proof}

\begin{lemma}
\label{LM12a} 
Let $U\in H^1(\Oet)$ be harmonic in $\Oet$. Then
\begin{equation}\label{M8}
\int_{\Oet}|U|^2dx\leq C\int_{\Get}|U|^2dS.
\end{equation}
\end{lemma}

\begin{proof} We start from the estimate
\begin{equation}
\label{M8a}
\int_{\Get}N(U)^2 \, dS \leq C \int_{\Get} |U|^2 \, dS,
\end{equation}
which is true because $U$ is harmonic in $\Oet$. Let
\[
\Omega^\delta=\{ x\in \Oet\; :\; \dist (x,\Get)<\delta\}.
\]
Now, using the Caccioppoli-inequality
\begin{equation*}
\int_{\Oet \setminus \Omega^{\delta}} |\nabla U|^2 \, dx 
	\leq C \delta^{-2} \int_{\Omega^{\delta}} |U|^2 \, dx,
\end{equation*}
we obtain that
\begin{equation}
\label{M8b}
\int_{\Oet \setminus \Omega^{2\delta}} |U|^2  \, dx 
	\leq C \delta^{-2} \int_{\Omega^{2\delta}} |U|^2 \, dx,
\end{equation}
which together with the fact that the integral in the right-hand side of~(\ref{M8b}) is estimated by a constant times 
the left-hand side of~(\ref{M8a}), provided that~$\delta$ is sufficiently small independently of~$d$, proves
that inequality~(\ref{M8}) holds.
\end{proof}

Let us now prove the assertion from Section~\ref{s:prelim} that we can choose~$\eps$ in~(\ref{eq:sigma})
as the small parameter~$\eps = d$.

\begin{lemma}
\label{l:close}
Let~\/$\Oe$ and~\/$\Ot$ be Lipschitz domains from ${\mathcal L}(M,\C_1,\ldots,\C_N)$.
Then
\begin{equation}
\label{eq:close}
\int_{\Dd} |\Se u - \St u|^2 \, dx \leq C d \, \| u \|_{H^1(\Dd)}.
\end{equation}
for all~$u \in H^1_0(\Dd)$.
\end{lemma}

\begin{proof}
Let~$w = \Se u - \St u$. Then~$w \in H_0^1(D)$ vanishes outside~$\Oe \cup \Ot$ and satisfies~$\Delta w = 0$ in~$\Oet$.
We start by proving that
\begin{equation}
\label{eq:est_w_cap}
\int_{\Oet} |w|^2 \, dx \leq C d \, \| w \|_{H^1(\Dd)}^2.
\end{equation}
By Lemma~\ref{LM12a}, the left-hand side of~(\ref{eq:est_w_cap}) is bounded by the squared~$L^2$-norm
of~$w$ along the boundary~$\Get$. Since~$w$ vanishes outside~$\Oe \cup \Oe$, 
it is clear that
\[
|w(x)| \leq \int_{g_{12}}^{h_{12}} |\partial_s w(x',s)| \, ds, \quad x \in \Get,
\]
where~$h_{12} = \max(h_1,h_2)$ and~$g_{12} = \min(h_1,h_2)$. Applying H\"older's inequality, we obtain that
\[
w^2 \leq C d \int_{g_{12}}^{h_{12}} |\nabla w(x',s)|^2  \, ds,
\] 
which implies that~(\ref{eq:est_w_cap}) holds. 

Let us now prove an inequality of Poincar\'{e}-Friedrichs type:
\begin{equation}
\label{eq:PF_w}
\int_{A} w^2 \, dx \leq C d^2 \int_{A} |\nabla w|^2 \, dx,
\end{equation}
where~$A = (\Oe \setminus \Ot) \cup (\Ot \setminus \Oe)$. 
Indeed, since~$w = 0$ outside~$\Oe \cup \Ot$, we obtain that
\[
\begin{aligned}
\int_{g_{12}}^{h_{12}} w^2(x', x_n) \, dx_n &= 
	2 \int_{g_{12}}^{h_{12}} \int_{g_{12}}^{x_n} w(x', s) \partial_s w(x',s) ds \, dx_n \\
&\leq C \int_{g_{12}}^{h_{12}} \left( \int_{g_{12}}^{x_n} w(x', s)^2 \, ds\right)^{1/2} 
		\left( \int_{g_{12}}^{x_n} |\nabla w(x',s)|^2 \, ds \right)^{1/2}  \, dx_n \\
&\leq C d \left( \int_{g_{12}}^{h_{12}} w(x', s)^2 \, ds\right)^{1/2} 
		\left( \int_{g_{12}}^{h_{12}} |\nabla w(x',s)|^2 \, ds \right)^{1/2}, \\
\end{aligned}
\]
where we used H\"older's inequality and the fact that~$|h_{12} - g_{12}| \leq C d$. This implies that
\[
\int_{B_{r_k}'(0)} \int_{g_{12}}^{h_{12}} w^2(x', x_n) \, dx_n dx'  \leq 
C d^2 \int_{B_{r_k}'(0)} \int_{g_{12}}^{h_{12}} |\nabla w(x', x_n)|^2 \, dx_n \, dx',
\]
which proves that~(\ref{eq:PF_w}) holds.
\end{proof}

\begin{lemma}
\label{L21a}
Let~$\vp \in \Xj[m]$ with~$\| \vp \|_{H^1(\Oe)} = 1$. Then
\begin{equation}
\label{M20}
\int_{\Get} \vp^2 dS \leq C d^2\;\;\mbox{and}\;\; \int_{\Get} |\nabla\vp|^2 dS \leq C.
\end{equation}
Moreover,
\begin{equation}
\label{M21}
\int_{\Oe \setminus \Ot} |\nabla\vp|^2 dx \leq C d.
\end{equation}
\end{lemma}

\begin{proof} 
We represent~$\vp$ as~$\varphi = U + V$, where
\[
-\Delta U = \lm \vp \;\; \mbox{in $\Dd$ and} \;\; U = 0 \;\; \mbox{on $\partial \Dd$}
\]
and
\[
\Delta V = 0 \;\; \mbox{on $\Oe$ and} \;\; V = -U \;\; \mbox{on $\Ge$.}
\]
Since $\vp \in H^1(\Oe)$, we have that~$U \in H^3(D)$ and that its norm is bounded 
by~$C\|\vp\|_{H^1(\Oe)}$. Using an embedding theorem for anisotropic spaces
(see Section~10 of~\cite{BIN}), we get~$U \in L^{\infty,{\bf 2}}({\mathcal C}_k)$
and~$|\nabla U|\in L^{\infty,{\bf 2}}({\mathcal C}_k)$, 
where~${\bf 2}=(2,\ldots,2)$ is an~$(n-1)$-vector, and~${\mathcal C}_k$, $k=1,\ldots,N$, are 
the same cylinders as in Section~\ref{SK1}. In the local coordinates $(y',s)$,
this means that
\begin{equation}
\label{M20a}
\int_{B'_{r_k}(0)} \sup_{|s| < \delta_k} |U(y',s)|^2 \, dy' < \infty
\end{equation}
and
\begin{equation}
\label{M20b}
\int_{B'_{r_k}(0)} \sup_{|s| < \delta_k} |\nabla U(y',s)|^2 \, dy' < \infty.
\end{equation}
Moreover, the above integrals are estimated by~$C\| \vp \|^2_{H^1(\Oe)}$.

Furthermore, since~$V$ and~$\nablat V$ belong to $L^2(\partial\Oe)$, we have that
\begin{equation}
\label{M20c}
\int_{\Ge} (N(V)^2 + N(\nabla V)^2) dS \;\; \mbox{is bounded and $\leq C\| \vp \|^2_{H^1(\Oe)}$}.
\end{equation}
Inequalities~(\ref{M20b}) and~(\ref{M20c}) imply
\begin{equation}
\label{M21a}
\int_{B'_{r_k}(0)} \sup_{|s| < \delta_k} |\nabla \vp(y',s)|^2 \, dy' < \infty.
\end{equation}
Therefore
\[
\begin{aligned}
\int_{B'_{r_k}(0)} \vp^2(y',h_{12}(y')) \, dy'
	&= 2\int_{B'_{r_k}(0)} \int_{h_1}^{h_{12}} |\partial_s \vp(y',s)| \, |\vp(y',s)| ds \, dy'\\
&\leq C d^2 \int_{B'_{r_k}(0)} \sup_{|s| < \delta_k} |\nabla \vp(y',s)|^2 \, dy'
\end{aligned}
\]
and the right hand side is bounded by~$Cd^2 \| \vp \|_{H^1(\Oe)}$ because of~(\ref{M20b}) 
and~(\ref{M20c}). This proves the first inequality in~(\ref{M20}). 
The second inequality in~(\ref{M20}) and~(\ref{M21}) follows from~(\ref{M21a}).
\end{proof}

\noindent The next lemma asserts that similar estimates are valid for the function $\Rvp$.

\begin{lemma}
\label{L21b}
Let~$\vp \in \Xj[m]$ with~$\| \vp \|_{H^1(\Oe)} = 1$. Then
\begin{equation}
\label{M20k}
\int_{\Get} \Rvp^2 \, dS \leq Cd^2 \;\; \mbox{and} \;\; \int_{\Get}|\nabla \Rvp|^2 \, dS \leq C.
\end{equation}
Moreover,
\begin{equation}
\label{M21l}
\int_{\Ot \setminus \Oe} |\nabla \Rvp|^2 \, dx \leq C d.
\end{equation}
\end{lemma}

\begin{proof}
We represent $\Rvp$ as $\Rvp = U + V$, where~$U$ is the same as in the proof of the previous lemma, and
\[
\Delta V=0 \;\; \mbox{in $\Ot$ and} \;\; V = -U \;\; \mbox{on $\Gt$.}
\]
The remaining part of the proof is analogous with the proof of Lemma~\ref{L21a}.
\end{proof}

\begin{corollary} Let~$\vp \in \Xj[m]$ with~$\| \vp \|_{L^2(\Oe)} = 1$. Then
\begin{equation}
\label{M12c}
\rho=O(d^2)\,.
\end{equation}
\end{corollary}

\begin{proof}
Since $\T\vp$ is harmonic in~$\Oet$, Lemma~\ref{LM12a} implies that
\[
\| \T \vp \|_{L^2(\Oet)}^2 \leq C \int_{\Get} \vp^2 \, dS,
\]
where we used the relation~$\T\vp = \vp$ in~$\Oe \setminus \Ot$.
By~(\ref{M20}) we get
\begin{equation}
\label{M21aa}
\| \T \vp \|_{L^2(\Oet)}^2 \leq C d^2.
\end{equation}
From (\ref{M21}) we derive
\[
\int_{\Oe \setminus \Ot} \vp^2 \, dx = O(d^3),
\]
which together with~(\ref{M21aa}) gives
\begin{equation}
\label{M21ab}
\| \T \vp \|_{L^2(\Oe)}^2 \leq C d^2.
\end{equation}

Since $\Vp = \vp - \Rvp$, using Lemmas~\ref{L21a} and~\ref{L21b}, we get
\begin{equation}
\label{M21ac}
\| \Vp \|_{L^2(\Ot)}^2 \leq C d^2,
\end{equation}
and from~(\ref{eq:est_VpVs}) and Lemma~\ref{L21b}, we obtain
\begin{equation}
\label{QQq}
\int_{\Ot}|\nabla \Vp|^2 \, dx \leq Cd.
\end{equation}
By the definition of~$\rr$ in~(\ref{eq:rho}), estimates~(\ref{M21ab}),~(\ref{M21ac}) and~(\ref{QQq}),
and finally Lemma~\ref{l:close}, we obtain that~(\ref{M12c}) holds.
\end{proof}

Another consequence of the above analysis is the following
estimate for the eigenvalues for Lipschitz domains.

\begin{corollary} Let~$\Oe$ and $\Ot$ be Lipschitz domains from~${\mathcal L}(M,\C_1,\ldots,\C_N)$.
Then estimate {\rm (\ref{JiJ})} holds.
\end{corollary}

\begin{proof} 
By~(\ref{eq:est_vpVs}) and~(\ref{eq:est_Vpvs}) combined with the first formula in~(\ref{M20}) 
and the last formula in~(\ref{M20k}), we get
\[
|\ip{\vp}{\Vs}| + |\ip{\vs}{\Vp}| = O(d).
\]
From~(\ref{eq:est_T}) and~(\ref{M20}), it follows that
\[
|\ip{\T\vp}{\T\vs}| = O(d)
\]
and
\[
|\ip{\S\vp}{\S\vs} - \ip{\vp}{\vs}| = O(d).
\]
Taking into account~(\ref{QQq}), we get $\tk = O(d)$, 
which together with~(\ref{M12c}) leads to~(\ref{JiJ}).
\end{proof}

\section{Two lemmas}

In the case that the boundary~$\Ge$ of~$\Oe$ is of class~$C^1$ (or~$C^{1,\alpha}$), it is more convenient to use 
another equivalent (for small $d$) description of the closeness of~$\Oe$ and~$\Ot$:
there exists a positive constant~$\delta_0$ such that for every~$P \in \Ge$,
in a cartesian coordinate system~$y=(y',y_n)$ with the center at~$P$ and the tangent plane to 
the boundary~$\Ge$ at~$P$ given by $y_n = 0$, the domain is given by
\[
\Oe \cap {\mathcal C}_{\delta_0,\delta_0}(P) = 
	\{ y \in {\mathcal C}_{\delta_0,\delta_0}(P) \,:\, y_n > h_1(y'), \, y' \in B'_{\delta_0}(P) \},
\]
where
\begin{equation}
\label{M31a}
{\mathcal C}_{\delta,\sigma} = \{ y  \, : \,  y' \in B'_{\delta}(P), \, |y_n| < \sigma \}.
\end{equation}
If $\sigma=\delta$ we shall use the notation
${\mathcal C}_{\delta}$ for ${\mathcal C}_{\delta,\delta}$.  The function $h_1 \in C^1$ satisfies
\begin{equation}\label{M23a}
h_1(0) = 0 \;\; \mbox{and} \;\; \nabla_{y'}h_1(0) = 0.
\end{equation}
The analogous representation is valid for $\Ot$ (certainly in this case we do not assume (\ref{M23a})) and
\begin{equation}
\label{M23b}
\sup_{B'_{\delta_0}(P)} |h_2 - h_1| = O(d) \;\; \mbox{and} \;\; 
	\sup_{B'_{\delta_0}(P)} |\nabla_{y'}(h_2 - h_1)| = o(1).
\end{equation}
In the case of a~$C^{1,\alpha}$-perturbation, we assume that~$h_1$ is of class $C^{1,\alpha}$ and the 
relations in~(\ref{M23b}) are replaced by
\begin{equation}
\label{M23c}
\sup_{B'_{\delta_0}(P)} |h_2 - h_1| = O(d) \;\; \mbox{and} \;\; 
	\sup_{B'_{\delta_0}(P)} |\nabla_{y'}(h_2 - h_1)| = O(d^\alpha).
\end{equation}
As before we use the notation $h_{12}=\max (h_1,h_2)$, which defines the domain~$\Oet$ in the local 
coordinates in~${\mathcal C}_{\delta_0,\delta_0}(P)$.

In what follows we shall use the following embedding assertion. 
If~$u\in H^1(\Dd)$ and the cylinder~${\mathcal C}_{\delta,\sigma}$ given by~(\ref{M31a})
in a certain cartesian coordinate system belongs to~$\Dd$, then, by Theorem~10.2 in~\cite{BIN},

\begin{equation}
\label{M31b}
\int_{B'_\delta} \sup_{|y_n| < \sigma} |u(y',y_n)|^2 \, dy' \leq C \| u \|^2_{H^1(\Dd)}.
\end{equation}

\begin{lemma}
\label{l:Phi}
Let~$\Oe$ and~$\Ot$ satisfy the~$C^1$-assumption.
Then the solution to~$\Delta U = f$ in~$\Oe$ and~$U = g$ on~$\Ge$,~$f\in L^2(\Oe)$ and~$g \in L^2(\Ge)$, satisfies
\begin{equation}
\label{eq:Phi}
\begin{aligned}
\int_{\Get \cap {\mathcal C}_\delta(P) } |U|^2 dS {} & \leq 
	C \Big( \int_{\Ge \cap {\mathcal C}_{\delta}(P)} |g|^2 \, dS
		+\int_{\Oe\cap {\mathcal C}_{\delta}(P)} |f|^2 \, dx\Big)\\
	&+ C d \Big(\int_{\Oe} |f|^2 \, dx+\int_{\Ge} |g|^2 \, dS\Big),
\end{aligned}
\end{equation}
where $P \in\Ge$, $\delta \leq \delta_0$ and~$C$ is a constant independent of~$\delta$,~$f$ and~$g$.
\end{lemma}

\begin{proof}
Let $f_\delta = f$ in~${\mathcal C}_{\delta}$ and $0$ otherwise and similarly, let~$g_\delta=g$
on~$\Ge \cap {\mathcal C}_{\delta}$ and zero on the remaining part of the boundary.
Let~$U_\delta$ solve the problem~$\Delta U_\delta = f_\delta$ in~$\Oe$ and~$U_\delta = g_\delta$ on~$\Ge$. 
We represent~$U_\delta$ as $U_\delta = U_\delta^1 + U_\delta^2$, 
where~$\Delta U_\delta^1 =f_\delta$ in~$\Dd$, $U_\delta^1 = 0$ on $\partial \Dd$ 
and~$\Delta U_\delta^2 = 0$ in~$\Oe$, $U_\delta = g_\delta - U_\delta^1$ on $\Ge$. Then
\begin{equation}
\label{M4a}
U_\delta^1 \in H^2(\Dd) \;\; \mbox{and} \;\; \| U_\delta^1 \|_{H^2(\Dd)} \leq C\| f_\delta \|_{L^2(\Oe)}.
\end{equation}
Then, according to, e.g.,~\cite{Kenig1994},~$U_\delta^2$ satisfies
\begin{equation}
\label{M4b}
\int_{\Ge} |N(U_\delta^2)|^2 \, dS \leq  C \Big( \int_{\Ge } |g_\delta|^2 \, dS + \int_{\Oe} |f_\delta|^2 \, dx \Big)
\end{equation}
and it follows from~(\ref{M4a}) and~(\ref{M4b}) that
\begin{equation}
\label{eq:N_Phz_est}
\int_{\Get} |U_\delta|^2 \, dS \leq
	C \Big(\int_{\Ge} |g_\delta|^2 \, dS + \int_{\Oe} |f_\delta|^2 \, dx \Big).
\end{equation}

Let~$U_r$ be the solution to~$ \Delta U_r = f - f_\delta$ in~$\Oe$
and~$U_r = g - g_\delta$ on~$\Ge$.
Since~$U_r = 0$ on~$\Ge \cap {\mathcal C}_{\delta}$, we have (compare with the proof of Lemma \ref{L21a})
\begin{equation}
\label{eq:N_Phr_est}
\begin{aligned}
\int_{\Get \cap {\mathcal C}_{\delta}} | U_r |^2 \, dS {} & \leq 
	C \int_{\Ge \cap {\mathcal C}_{\delta}} \int_{h_1}^{h_{12}} 
		|\nabla U_r(x',s)| \, |U_r(x',s)| \, ds \, dS(x')\\
	{} & \leq C d \int_{\Ge \cap {\mathcal C}_{\delta}} \int_{h_1}^{h_{12}} |\nabla U_r(x',s)|^2 \, ds \, dS(x').
\end{aligned}
\end{equation}
This implies that
\[
\int_{\Get \cap {\mathcal C}_{\delta}} | U_r |^2 \, dS \leq C d \int_{\Oe \setminus \Ot} |\nabla U_r|^2 \, dx.
\]
Now applying the estimate
\[
\int_{\Oe} |\nabla U_r |^2 \, dx \leq C \Big( \int_{\Oe} |f|^2 \, dx + \int_{\Ge} |g|^2 \, dS(x') \Big)
\]
and using~(\ref{eq:N_Phz_est}), we arrive at~(\ref{eq:Phi}).
\end{proof}

The estimates in Lemmas~\ref{l:VpVs}, \ref{l:vpVs} and \ref{l:est_T} contain the 
terms~$\nablat \Vp$ and~$\nablat \vp$.
In the next lemma we present estimates for such forms in terms of the small parameter~$d$.

\begin{lemma}
\label{LM28a}
Let~$\vp \in \Xj[m]$ with~$||\varphi||_{H^1(\Oe)}=1$. We have
\begin{equation}
\label{M25a}
\int_{\Get} |\nablat \vp|^2 \, dS = o(1) \;\; \mbox{and} \;\; \int_{\Get} |\nablat \Vp|^2 \, dS = o(1)
\end{equation}
in the~$C^1$-case and
\begin{equation}
\label{M25b}
\int_{\Get} |\nablat \vp|^2 \, dS = O(d^{2\alpha}) \;\; \mbox{and} \;\; \int_{\Get} |\nablat \Vp|^2 \, dS = O(d^{2\alpha})
\end{equation}
in the~$C^{1,\alpha}$-case.
\end{lemma}

\begin{proof}
Consider a cylinder~${\mathcal C}_{\delta}(P)$ for a certain~$\delta=O(d)$ and~$P \in \Ge$. 
Then the directions $y_j$,~$j=1,\ldots,n-1$, are tangent to~$\Ge$ at~$P$. 
We choose a direction~$y_j$ for a certain~$j=1,\ldots,n-1$, and put~$u_j=\partial_{y_j} \vp$.
Then
\[
-\Delta u_j = \lm u_j \;\; \mbox{in $\Oe$ and} \;\; u_j = g := \partial_{y_j}\vp \;\; \mbox{on $\Ge$.}
\]
The function~$g$ belongs to~$L^2(\partial\Oe)$ and since~$\partial_{y_j} \vp(y',h_1(y')) = 0$, we have
\[
g(y',h_1(y')) = -\partial_{y_n}\vp(y',h_1(y')) \partial_{y_j} h_1(y').
\]
Using that $\partial_{y_n}\vp(y',h_1(y')) \in L^2(\Ge)$, we obtain that
\[ 
\int_{B'_\delta} g^2 \, dy' = o(1) \int_{B'_\delta} |\nabla\vp|^2 \, dS.
\]
Noting that
\[
\int_{\Oe \cap {\mathcal C}_{\delta}(P)} \vp^2 \, dx \leq 
	Cd^2 \int_{\Oe \cap {\mathcal C}_{\delta}} |\nabla\vp|^2 \, dx
\]
and applying Lemma \ref{l:Phi}, we get
\[
\int_{\Get \cap {\mathcal C}_{\delta}(P)} u_j^2 \, dS = o(1) \int_{B'_\delta} |\nabla\vp|^2 \, dS
	+ O\Big( d^2 \int_{\Oe \cap {\mathcal C}_{\delta}} |\nabla\vp|^2 \, dx \Big).
\]
This implies the first inequality in (\ref{M25a}).

In the~$C^{1,\alpha}$-case, we have
\[
\int_{B'_\delta} g^2 \, dy' \leq C \delta^{2\alpha} \int_{B'_\delta} |\nabla\vp|^2 \, dS.
\]
Repeating the above proof for the~$C^1$-case, we arrive at the first inequality of~(\ref{M25b}).

Since~$\Vp = \vp - \Rvp$, where~$-\Delta \Rvp = \lm \vp$ in~$\Ot$ and~$\Rvp = 0$ on~$\Gt$, and
the required inequality for~$\vp$ is proved already, it is sufficient to prove the 
inequalities for~$\Rvp$ only. However, since the boundary value problem for~$\Rvp$ is similar
to that for~$\vp$, the proof can be carried out in the same way as for~$\vp$.
\end{proof}

\section{Proof of Theorems \protect\ref{i:t:holder_boundary} and \protect\ref{i:t:C1_boundary}}
By~(\ref{eq:est_vpVs}) and~(\ref{eq:est_Vpvs}) combined with Lemma~\ref{LM28a}, we obtain that
\[
|\ip{\vp}{\Vs}|+|\ip{\vs}{\Vp}| = \left\{ \begin{array}{ll}
	o(d) 		& \mbox{for the~$C^1$-case}\\
	O(d^{1+\alpha}) & \mbox{for the~$C^{1,\alpha}$-case}
\end{array} \right. \;\; \mbox{as $d\to 0$.}
\]
Moreover, due to~(\ref{eq:est_T}) and Lemma~\ref{LM28a},
\[
|\ip{\S\vp}{\S\vs} - \ip{\vp}{\vs}| = \left\{ \begin{array}{ll}
	o(d) 		& \mbox{for the~$C^1$-case}\\
	O(d^{1+\alpha}) & \mbox{for the~$C^{1,\alpha}$-case}
\end{array} \right. \;\; \mbox{as $d\to 0$.}
\]
Having in mind these relations and~(\ref{M12c}), we can write formula~(\ref{M28s}) as
\begin{equation}
\label{M28t}
\um[k]^{-1} = \lm[m]^{-1} + \tk + \left\{ \begin{array}{ll}
	o(d) 		& \mbox{for the~$C^1$-case}\\
	O(d^{1+\alpha}) & \mbox{for the~$C^{1,\alpha}$-case}
\end{array}\right., \; \; k = 1,2,\ldots,\Jm, 
\end{equation}
where~$\tau = \tk$ is an eigenvalue of the problem
\begin{equation}
\label{M28u}
\lm[m]^{-1} \Big(\int_{\Ot \setminus \Oe} \nabla \Rvp \cdot \nabla \Rvs \, dx
		- \int_{\Oe \setminus \Ot} \nabla \vp \cdot \nabla \vs \, dx \Big)
= \tau \ip{\vp}{\vs},
\end{equation}
where~$\vp \in \Xj[m]$ and~$\tau_1,\ldots,\tau_{J_m}$ run through all eigenvalues of~(\ref{M28u})
counted with multiplicity.

To obtain results for integrals over the domains~$\Oe \setminus \Ot$ and~$\Ot \setminus \Oe$ 
in the case of~$C^1$ or~$C^{1,\alpha}$ boundaries
phrased in terms of boundary integrals, the following two Lemmas will be helpful.

\begin{lemma}
\label{l:lip_boundaryint}
Let~$\Oe$ and~$\Ot$ satisfy the~$C^1$-assumption from Section~\ref{SK1}.
Suppose that~$g \in L^2(\Ge)$, $f \in L^2(\Oe)$ and that~$U$ solves~$\Delta U = f$ in~$\Oe$ and~$U = g$ on~$\Ge$.
Then
\begin{equation}
\label{eq:lip_boundaryint}
\int_{\Oe \setminus \Ot} |U|^2 \, dx = -\int_{\Ge} \sigma_{-} |g|^2 \, dS + o(d),
\end{equation}
as~$d \rightarrow 0$. Here,~$\sigma_{-} = \max (0,-\sigma)$, where~$\sigma$ is the function defined by~{\rm (\ref{KK1a})}. 
If, in addition,~$\Oe$ and~$\Ot$ satisfy the~$C^{1,\alpha}$-assumption from Section~\ref{SK1},~$g \in C^{0,\alpha}(\Ge)$
and~$f\in L^\infty(\Oe)$, then~{\rm(}\ref{eq:lip_boundaryint}{\rm)} is valid with~$o(d)$ replaced by~$O(d^{1+\alpha})$. 
\end{lemma}

\begin{proof} 
We represent~$U$ as~$U = u + v$, where
\[
\Delta u = f\;\; \mbox{in $\Dd$ and} \;\; u = 0 \;\; \mbox{on $\partial \Dd$}
\]
and
\[
\Delta v = 0\;\; \mbox{in $\Oe$ and} \;\; v = g - u \;\; \mbox{on $\Ge$.}
\]
Then~$u \in H^2(\Dd)$ with~$\|u\|_{H^2(\Dd)} \leq C\|f \|_{L^2(\Oe)}$ and
\begin{equation}
\label{M5a}
\int_{\Ge} \N(v)^2 \, dS \leq C( \|g\|^2_{L^2(\Ge)} + \|f\|^2_{L^1(\Oe)}).
\end{equation}
These estimates lead to
\begin{equation*}
\begin{aligned}
\int_{B'_\delta} \int_{h_1(y')}^{h_{12}(y')} |U(y',y_n)|^2 \, dy 
	& {} \leq C d \int_{B'_\delta} \sup_{h_1 < y_n < h_{12}} U^2(y',y_n) \, dy'\\
	& {} \leq C d \, ( \|g\|^2_{L^2(\Ge)} + \|f\|^2_{L^2(\Oe)} ),
\end{aligned}
\end{equation*}
which implies that
\begin{equation}
\label{M30a}
\int_{\Oe \setminus \Ot} |U|^2 \, dx \leq C d \, ( \|g\|^2_{L^2(\Ge)} + \|f\|^2_{L^2(\Oe)} ).
\end{equation}

Let us choose an approximating sequence~$g_j \in H^1(\Ge)$, $j=1,2\ldots$, of the function~$g$, 
i.e., we assume that $\epsilon_j := \| g - g_j \|_{L^2(\Ge)} \to 0$ when~$j \to \infty$. 
If we denote by~$U_j$ the solution to
\[
\Delta U_j = f \;\; \mbox{in $\Oe$ and} \;\; U_j = g_j \;\; \mbox{on $\Ge$,}
\]
then by~(\ref{M30a}),
\begin{equation}
\label{feb23b}
\int_{\Oe \setminus \Ot} |U-U_j|^2 \, dx \leq C d \, \epsilon_j^2.
\end{equation}
We represent~$U_j$ as~$U_j = u + v_j$, where~$\Delta v_j = 0$ in $\Oe$ and $v_j = g_j - u$ on~$\Ge$. Then
\[
\int_{B'_\delta} \sup_{h_1 < y_n < h_{12}} |\partial_{y_n} U_j(y',y_n)|^2 \, dy' 
	\leq C(\|g_j\|^2_{H^1(\Ge)} + \|f\|^2_{L^2(\Oe)})
\]
since~$g_j \in H^1(\Ge)$ and~$u \in H^2(\Oe)$ (see (\ref{M31b})).

Observing that
\[
\begin{aligned}
\int_{B'_{\delta}} \int_{h_1}^{h_{12}} U_j^2(y',y_n) \, dy_n \, dy' 
	= {} & \int_{B'_{\delta}} \int_{h_1}^{h_{12}} U_j^2(y',h_1(y')) \, dy_n \, dy'\\
	& + 2 \int_{B'_{\delta}} \int_{h_1}^{h_{12}} \int_{h_1}^{y_n} \partial_s U_j(y',s) \,U_j(y',s) \, ds \, dy_n \, dy',
\end{aligned}
\]
we thus obtain that
\begin{equation}
\label{feb23a}
\Big| \int_{B'_{\delta}} \int_{h_1(y')}^{h_{12}(y') } (U_j^2(y',y_n) - g_j^2(y')) \, dy_n \, dy' \Big|
	\leq Cd^2 \int_{B'_{\delta}} \widehat{N}(U_j) \widehat{N}(\partial_{y_n} U_j) \, dy',
\end{equation}
where
\[
\widehat{N}(w)(y') = \sup_{|y_n| < \delta} |w(y',y_n)|.
\]
From~(\ref{feb23a}) it follows that
\begin{equation}
\label{M31c}
\Big| \int_{B'_{\delta}} \int_{h_1(y')}^{h_{12}(y')} U_j^2(y',y_n) \, dy_n dy'  - 
	\int_{B'_{\delta}} \bigl( h_2(y') - h_1(y') \bigr) g_j^2(y') \, dy' \Big| \leq C_j d^2
\end{equation}
with a certain constant depending on~$j$ but independent of~$d$.

Now, since
\[
\nu(y') = \frac{(\nabla h_1(y'),-1)}{\sqrt{1+|\nabla h_1(y')|^2}},
\]
we have (due to the~$C^1$-assumption)
\[
\sigma(y') = \frac{h_1(y') - h_2(y')}{\sqrt{1+|\nabla h_1(y')|^2}} + o(d) \;\; \mbox{as $d\to 0$,}
\]
which together with~(\ref{M31c}) implies that
\[
\Big| \int_{B'_\delta} \int_{h_1(y')}^{h_{12}(y')} U_j^2(y',y_n) \, dy_n \, dy' - 
		\int_{B'_\delta} \sigma(x') |g_j(x')|^2 \, dS(x') \Big| = o(d),
\]
where~$dS(x') = \sqrt{1+|\nabla h_1(y')|^2} \, dy'$. 
This together with~(\ref{feb23b}) implies that~(\ref{eq:lip_boundaryint}) holds.

We now consider the~$C^{1,\alpha}$-case. Indeed, the solution (weak solution) then 
belongs to~$C^{1,\alpha}(\Oe)$ and
\[
\| U \|_{C^{1,\alpha}(\Oe)} \leq C ( \| g\|_{C^{1,\alpha}(\Ge)} + \|f\|_{L^\infty(\Oe)} ) \mbox{;}
\]
see for instance Agmon et al.~\cite{Agmon1959}.
In this case, we can write
\begin{equation*}
\Big| \int_{B'_{\delta}} \int_{h_1(y')}^{h_{12}(y')} U^2(y',y_n) \, dy_n \, dy' - 
		\int_{B'_{\delta}} \int_{h_1(y')}^{h_{12}(y')} U^2(y',h_1(y')) \, dy_n \, dy' \Big| \leq C d^{1+\alpha}
\end{equation*}
because of the H\"older continuity of~$U$. Furthermore, it is clear that
\[
\sigma (y') = \frac{h_1(y') - h_2(y')}{\sqrt{1+|\nabla h_1(y')|^2}} + O(d^{1+\alpha}) \;\; \mbox{as $d\to 0$,}
\]
by the~$C^{1,\alpha}$-assumption. This proves that~(\ref{eq:lip_boundaryint}) holds with the remainder~$o(d)$ replaced
by~$O(d^{1+\alpha})$.
\end{proof}

\noindent Analogously with Lemma~\ref{l:lip_boundaryint}, we can prove the following result.

\begin{lemma}
\label{l:lip_boundaryint2}
Let~$\Oe$ and~$\Ot$ be the same domains as in the previous lemma.
Suppose that~$g \in L^2(\partial \Ot)$, $f\in L^2(\Ot)$ and that~$U \in H^1(\Ot)$ solves~$\Delta U = f$ in~$\Ot$ and~$U = g$ on~$\partial \Ot$.
Then
\begin{equation}
\label{M28a}
\int_{\Ot \setminus \Oe} |U|^2 \, dx = \int_{\Ge \cap \Ot} \sigma_{+} \, |U|^2 \, dS + o(d)\;\;\mbox{as~$d \rightarrow 0$}
\end{equation}
in the~$C^1$-case, where~$\sigma_{+} = \max(\sigma, 0)$. In the~$C^{1,\alpha}$-case, we assume in addition
that~$g\in C^{0,\alpha}(\Gt)$ and~$f\in L^\infty(\Ot)$, and the remainder~$o(d)$ can be replaced by~$O(d^{1+\alpha})$
in~{\rm(}\ref{M28a}{\rm)}.
\end{lemma}

\section{The end of the proof of Theorems \protect\ref{i:t:holder_boundary} and \protect\ref{i:t:C1_boundary}}

The function~$U = \partial_{x_j}\vp$ satisfies
\[
\Delta U = f := -\lm \partial_{x_j} \vp \;\; \mbox{in $\Oe$ and} 
	\;\; U = g := \partial_{x_j}\vp \;\; \mbox{on $\Ge$.}
\]
Clearly, $f \in L^2(\Oe)$ and~$g \in L^2(\Ge)$ in the~$C^1$-case, 
and in the~$C^{1,\alpha}$-case, we have~$f \in C^{0,\alpha}(\Oe)$ and~$g \in C^{0,\alpha}(\Ge)$.
Applying Lemma~\ref{l:lip_boundaryint}, we obtain
\begin{equation}
\label{M31k}
\int_{\Oe \setminus \Ot} \partial_{x_j} \vp \, \partial_{x_j} \vs \, dx = 
	-\int_{\Ge} \sigma_{-} \, \partial_{x_j} \vp \, \partial_{x_j} \vs \, dS + o(d)
\end{equation}
in the~$C^1$-case and with the remainder~$o(d)$ replaced by~$O(d^{1+\alpha})$ in the~$C^{1,\alpha}$-case. 
We note that in order to apply Lemma~\ref{l:lip_boundaryint},
it is useful to use the relation
\begin{equation}
\label{M31n}
\partial_{x_j} \vp \, \partial_{x_j} \vs = \frac{1}{2} \bigl( \partial_{x_j}\vp + \partial_{x_j}\vs\bigr)^2 - 
	\frac{1}{2} \partial_{x_j}\vp^2 - \frac{1}{2}\partial_{x_j}\vs^2
\end{equation}
since formula~(\ref{eq:lip_boundaryint}) only contains squares of functions. 
Equality~(\ref{M31k}) implies that
\begin{equation}
\label{M31l}
\int_{\Oe \setminus \Ot} \nabla\vp \cdot \nabla\vs \, dx = -\int_{\Ge} \sigma_{-} \nabla\vp \cdot \nabla\vs \, dS + o(d)
\end{equation}
for the~$C^1$-case and with the remainder replaced by~$O(d^{1+\alpha})$ in the~$C^{1,\alpha}$-case.

Furthermore, the function~$U = \partial_{x_j} \Rvp$ satisfies
\[
\Delta U = f := -\lm \partial_{x_j} \vp \;\; \mbox{in~$\Ot$ and} \;\; U = g := \partial_{x_j} \Rvp \;\; \mbox{on~$\Gt$.}
\]
Applying Lemma \ref{l:lip_boundaryint2}, we get
\begin{equation}
\label{M31m}
\int_{\Oe \setminus \Ot} \partial_{x_j}\Rvp \, \partial_{x_j} \Rvs \, dx = 
	\int_{\Ge} \sigma_{+} \, \partial_{x_j} \Rvp \, \partial_{x_j} \Rvs \, dS + o(d)
\end{equation}
in the~$C^1$-case, and in the~$C^{1,\alpha}$-case, the remainder~$o(d)$ is replaced by~$O(d^{1+\alpha})$.
Using that~$\partial_{x_j} \Rvp = \partial_{x_j} \vp - \partial_{x_j}\Vp$ and that
\[
\int_{\Get} |\nabla \Vp|^2 \, dS = o(1) \;\; \mbox{as $d\to 0$}
\]
in the~$C^1$-case and~$O(d^{2\alpha})$ in the~$C^{1,\alpha}$-case, we obtain
\begin{equation}
\label{M31r}
\int_{\Oe \setminus \Ot} \partial_{x_j} \Rvp \, \partial_{x_j} \Rvs \, dx = 
	\int_{\Ge} \sigma_{+} \, \partial_{x_j} \vp \, \partial_{x_j}\vs \, dS + o(d)
\end{equation}
in the~$C^1$-case, and with the remainder replaced by~$O(d^{1+\alpha})$ in the~$C^{1,\alpha}$-case.
Therefore
\begin{equation}
\label{M31p}
\int_{\Ot \setminus \Oe} \nabla \Rvp \cdot \nabla \Rvs \, dx = \int_{\Ge} \sigma_{+} \, \nabla\vp \cdot \nabla\vs \, dS + o(d)
\end{equation}
in the~$C^1$-case, and with the remainder replaced by~$O(d^{1+\alpha})$ in the~$C^{1,\alpha}$-case.

Now, applying formulas~(\ref{M31l}) and~(\ref{M31p}) to~(\ref{M28u}), we prove 
Theorems~\ref{i:t:holder_boundary} and~\ref{i:t:C1_boundary}.

\section{Counterexamples}\label{SA4a}

Let~$\Oe$ be the domain in $\mathbf{R}^2$ given by
\[
\Oe = \{ (x,y) \,:\, 0 < x < T,\, 0 < y < R \}.
\]
The domain~$\Ot$ is given by
\[
\Ot = \{ (x,y) \,:\, 0 < x < T, \, d\eta(x/\delta) < y < R \},
\]
where~$\eta$ is a positive, periodic, $C^1$-function such that~$\eta(X+1)=\eta(X)$. 
We assume that~$\delta$ and~$d$ are small parameters, $d \leq \delta$, and that $T/\delta = N$, 
where~$N$ is a large integer. We will consider three cases:~$\delta = d$ is a Lipschitz perturbation, 
if~$\delta = o(d)$ we are dealing with a~$C^1$-perturbation and if~$\delta = d^{1-\alpha}$, 
the perturbation is of the class~$C^{1,\alpha}$.

Consider the eigenvalue problems
\[
-\Delta \vp = \lm \vp \;\; \mbox{in $\Oe$ and $\vp(x,0) = \vp(x,R) = 0$ for $0 < x < T$.}
\]
In addition, we assume that the function is periodic in $x$, i.e.
\[
\vp(0,y) = \vp(T,y) \;\; \mbox{and} \;\; \partial_x \vp(0,y) = \partial_x\vp(T,y) \;\; \mbox{for $0 < y < R$.}
\]
We are interested in a perturbation of the first eigenvalue~$\lm=\lambda_1$. 
Separating the variables, one can easily find that
\[
\lambda_1 = \frac{\pi^2}{R^2} \;\; \mbox{and} \;\; \vp = \varphi_1 = \frac{\sqrt{2}}{\sqrt{RT}}\sin\Big(\frac{\pi}{R}y\Big),
\]
where $\varphi_1$ is a corresponding eigenfunction normalized so that~$\|\varphi_1\|_{L^2(\Oe)} = 1$. 
Clearly,~$\lambda_1$ is a simple eigenvalue.

The perturbed problem is the following:
\[
-\Delta u = \mu u \;\; \mbox{in $\Ot$ and $u(x,d \eta(x/\delta)) = u(x,R)=0$ for $0 < x < T$.}
\]
Furthermore, one must add periodicity conditions on the parts of the boundary where~$x=0$ or~$x=T$.

In this case,~$\Ot \subset \Oe$, and therefore~$\Vp = 0$ and the formula in~(\ref{eq:main_asymp_D}) 
is reduced to
\[
\tau_1 = -\frac{1}{\lambda_1} \int_{\Oe} |\nabla\Phi|^2 \, dx \, dy
\]
and~(\ref{M28s}) becomes
\[
\mu_1^{-1} = \lambda_1^{-1} + \tau_1 + O(d^2).
\]
Here $\Phi = \varphi_1$ in $\Oe \setminus \Ot$ and
\[
\Delta \Phi = 0 \;\; \mbox{in $\Ot$ and} \;\; \Phi(x,R) = 0,\, 
	\Phi(x,d\eta(x/\delta)) = \varphi_1(x,d\eta(x/\delta)),\; 0 < x < T.
\]
Moreover, the periodicity condition must be satisfied on the parts of the boundary where~$x=0$ or~$x=T$. 
Since
\begin{equation}
\label{A3a}
\int_{\Oe} |\nabla\Phi|^2 \, dx\, dy = \int_{\Oe \setminus \Ot} |\nabla\varphi_1|^2 \, dx\, dy 
	+ \int_{\Ot} |\nabla\Phi|^2 \, dx \, dy
\end{equation}
and the first term in the right-hand side corresponds to the Hadamard term in~(\ref{A4a}) and 
has the order~$O(\delta)$, it suffices to analyze the second term in~(\ref{A3a}).

We will construct the function~$\Phi$ in~$\Ot$ using the following representation:
\[
\Phi(x,y) = d w_0(x,y) + d V_0(X,Y) + {\mathcal R}(x,y),
\]
where~$X=x/\delta$ and~$Y=y/\delta$. The function~$w_0$ solves the following boundary value problem in $\Oe$:
\[
\Delta w_0 = 0 \;\; \mbox{in $\Oe$ and} \;\; w_0(x,R) = 0, \;\; w_0(x,0) = c_0,
\]
where~$c_0$ is a constant to be determined later. As usual, periodicity conditions are valid on the vertical
parts of the boundary.
The solution to this problem is given by
\[
w_0 = \frac{c_0(R-y)}{R}.
\]
In order to write the boundary value problem for~$V_0$, we note that
\[
\varphi_1(x,y) = d C_1 \eta(X) + O(d^3) \;\;\; \mbox{when $y=d\eta(x/\delta)$}, \;\;
	C_1 = \frac{\sqrt{2}}{\sqrt{RT}}\frac{\pi }{R}.
\]
The function~$V_0$ is periodic with respect to~$X$ with the period~$1$ and solves the problem
\begin{equation}
\label{A3b}
\Delta_{X,Y} V_0 = 0 \;\;\mbox{in ${\mathcal D}$ and} \;\; V_0(X,d\eta(X)/\delta) = C_1\eta(X)-c_0,
\end{equation}
where
\[
{\mathcal D}=\{ (X,Y) \,:\, \eta(X) < Y < \infty, \, 0 < X < 1 \}.
\]
Furthermore, it is assumed that~$V_0$ decays exponentially as~$Y \to \infty$. 
As it is known, the problem~(\ref{A3b}) has a unique solution~$V_0$ in the class of periodic 
solutions subject to, e.g.,
\[
\int_{{\mathcal D}} |\nabla V_0|^2 \, dX \, dY + \int_0^1 \int_{\eta(X)}^1 |V_0|^2 \, dY \, dX < \infty,
\]
and this solution admits the asymptotic representation
\[
V_0(X,Y) = q + O(e^{-\pi Y}) \;\; \mbox{as $y \to \infty$.}
\]
To calculate the coefficient~$q$, one can introduce a special solution~${\mathcal V}(X,Y)$
which solves the to~(\ref{A3b}) corresponding homogeneous problem and has the asymptotics
\[
{\mathcal V}(X,Y) = Y + c_2 + O(e^{-\pi Y}).
\]
By applying the maximum principle we obtain that~${\mathcal V} > 0$ inside~${\mathcal D}$. 
If~$\delta=d$, then by Hopf's lemma
\[
\partial_{\nu} {\mathcal V}(X,\eta(X)) < 0 \;\; \mbox{for $0 \leq X \leq 1$.}
\]
In the case when~$\delta = o(d)$ ($\delta = O(d^{1-\alpha})$), 
${\mathcal V} = \widetilde{\mathcal V} + o(d)$ (${\mathcal V} = \widetilde{\mathcal V} + O(d^\alpha)$), 
where~$\widetilde{\mathcal V}$ solves the homogeneous problem
\[
\Delta_{X,Y} \widetilde{\mathcal V} = 0 \;\; \mbox{in $\widetilde{\mathcal D}$ and} \;\; 
	\widetilde{\mathcal{V}} (X,0) = 0,
\; \; \mbox{where~$\widetilde{\mathcal D} = (0,1) \times (0,\infty)$.}
\]
By Hopf's lemma,~$\partial_Y \widetilde{\mathcal V}(X,0) > 0$ for~$0 \leq X \leq 1$.
Now, let~$\mathcal{D}_c$ be the domain~$\mathcal{D}$ truncated at~$Y = c$ so that~$\eta(X) < Y < c$. 
Since~$V_0$ and~$\mathcal{V}$ are harmonic in~$\mathcal{D}$, it is clear that
\[
\begin{aligned}
\int_{0}^{1} \mathcal{V}(X, c) \partial_{Y} V_0(X,c) \, dX 
= {} & 
\int_{0}^{1} \partial_{Y} \mathcal{V}(X,c) V_0(X,c) \, dX\\
& + \int_0^1 \partial_\nu {\mathcal V}(X,\eta(X))(C_1\eta(X) - c_0)\sqrt{1+\eta'(X)^2} \, dX.
\end{aligned}
\]
Using the asymptotic representations above and letting~$c \to \infty$,
we obtain that
\[
q = -\int_0^1 \partial_\nu {\mathcal V}(X,\eta(X))(C_1\eta(X)-c_0)\sqrt{1+\eta'(X)^2} \, dX.
\]
The constant~$c_0$ is sought from the relation $q = 0$, which guarantees the exponential decay of the 
function~$V_0$. 

Finally, the remainder~${\mathcal R}$ satisfies
\[
\Delta {\mathcal R} = 0 \;\; \mbox{in $\Ot$,} \;\; {\mathcal R}(x,R) = O(e^{-\pi R/\delta}) \;\; 
	\mbox{and} \;\; {\mathcal R}(x,d\eta(x/\delta)) = O(d^3)
\]
and the differentiation of the boundary conditions will give a factor~$1/\delta$. 
This leads to the estimate
\[
\int_{\Ot}|\nabla {\mathcal R}|^2 \, dx\, dy = O(d^2).
\]
Therefore
\[
\begin{aligned}
\int_{\Ot} |\nabla\Phi|^2 \, dx \, dy = {} & \int_0^T \int_{d\eta(x/\delta)}^R 
	\Big(\Big| \frac{d}{\delta} \partial_{X} V_0(X,Y) \Big |^2+|d (w_0)'_y + \frac{d}{\delta} \partial_{Y} V_0(X,Y)|^2\Big) dy \, dx \\
	& + O(d^2).
\end{aligned}
\]
Since
\[
\begin{aligned}
\int_0^T \int_{d\eta(x/\delta)}^R \Big|\frac{d}{\delta} \partial_X V_0(X,Y)\Big |^2 dy \, dx = {} & 
	T \frac{d^2}{\delta} \int_0^1 \int_{d\eta(X)/\delta}^\infty |\partial_X V_0(X,Y)|^2 dY \, dX\\
	& + O(d^2)
\end{aligned}
\]
and
\[
\int_0^T \int_{d\eta(x/\delta)}^R |d (w_0)'_y + \frac{d}{\delta}\partial_Y V_0(X,Y)|^2 dy \, dx = 
	O\Big( \frac{d^2}{\delta} \Big) + O(d^2),
\]
we see that an additional term in the Hadamard formula will appear if~$\delta=d$ and that
the remainders~$o(d)$ and~$O(d^{1+\alpha})$, respectively, are sharp in
theorems~\ref{i:t:holder_boundary} and~\ref{i:t:C1_boundary}.


\def\bibname{References}

\end{document}